\documentclass{article}
\usepackage[utf8]{inputenc}
\usepackage{amsthm}
\usepackage{amsmath}
\usepackage{amsfonts}
\usepackage{latexsym}
\usepackage{geometry}
\usepackage{enumerate}
\usepackage{csquotes}
\usepackage{graphicx}
\usepackage{comment}
\usepackage{appendix}
\usepackage{hyperref}
\usepackage{bm}

\emergencystretch=\maxdimen
\def \W {\mathcal{W}}
\def \N {\mathcal{N}}

\def \bW {\overline{\mathcal{W}}}
\def \dist {{\rm dist}}
\DeclareMathOperator{\Rm}{Rm}
\DeclareMathOperator{\Ric}{Ric}

\DeclareMathOperator{\tr}{tr}
\makeatletter
\newcommand*{\rom}[1]{\rm {\expandafter\@slowromancap\romannumeral #1@}}
\makeatother

\def \V {\mathcal{V}}

\def \n {\textbf{n}}

\numberwithin{equation}{section}
\newtheorem{Theorem}{Theorem}[section]

\newtheorem{Corollary}[Theorem]{Corollary}
\theoremstyle{definition}

\newtheorem{Remark}[Theorem]{Remark}

\title{Volume growth estimates of gradient Ricci solitons \\ \smallskip 
\emph{\footnotesize Dedicated to Professor Peter Li on the occasion of his seventieth birthday}}
\author{Pak-Yeung Chan, Zilu Ma, and Yongjia Zhang}

\begin{document}

\maketitle

\begin{abstract}
In this paper, we survey the volume growth estimates for shrinking, steady, and expanding gradient Ricci solitons. Together with the known results, we also prove some new volume growth estimates for expanding gradient Ricci solitons.

\end{abstract}

\section{Introduction}

Ever since Hamilton \cite{Ha82} invented the Ricci flow, the study of gradient Ricci solitons has been an central element of this field. Up to this point, gradient Ricci solitons have already been extensively studied. Many classic results have been proved, and some played an important role in the whole field of Ricci flow. For instance, Ricci shrinkers can be classifed completely in dimension two (c.f. \cite{Ha88}, \cite{Ha95}, \cite{BM15}) and in dimension three (c.f. \cite[\S 26]{Ha95}, \cite[Lemma 1.2]{Per03a}, \cite{CCZ08},  \cite{NW08}, \cite{PW10}), and with additional conditions in higher dimensions (c.f. \cite{Na10}, \cite{CC12}, \cite{CC13}, \cite{LNW18}, \cite{N19}). Steady solitons also admit a classification in dimension two (c.f. \cite{Ha95}, \cite{BM15}). In dimension three, the only noncollapsed steady soliton is the Bryant soliton \cite{Br13}, although there is at least one collapsed example  \cite{Lai20}. In higher dimensions, more examples of noncollapsed steady solitons are found by Appleton \cite{Ap17} and Lai \cite{Lai20}. The case of expanding solitons is much less clear, since even in dimension 3, one may find a one-parameter family of Ricci expanders (\cite{RFV1}). 

When a classification is not available, one naturally turns to consider the qualitative or quantitative properties of a general soliton, in hope that the analytic properties may have some implications on the geometric properties, or vice versa. Research results in this respect are massive. It is therefore impossible to include every aspect of the field within one survey paper. Thus we shall focus on one important aspect---that of volume estimates. In this paper, we shall survey volume estimates of all three kinds of Ricci solitons. Furthermore, we prove some new volume growth estimates for Ricci expanders. In particular, our estimate in dimension three is sharp, since it shows that every three-dimensional expander with bounded scalar curvature and proper potential function has a volume growth rate no greater than $r^3$, and this rate is satisfied by the Gaussian expander.

\section{Volume estimates for shrinking gradient Ricci solitons}

In this section we consider a complete shrinking gradient Ricci soliton $(M^n,g,f)$ normalized in the way that
\begin{align}\label{normalized_shrinker}
    \Ric +\nabla^2f &=\tfrac{1}{2}g.
    \\\nonumber
    \int_M (4\pi)^{-\frac{n}{2}}e^{-f}dg&=1.
\end{align}
Because of the logarithmic Sobolev inequality of Bakry-Emery \cite{BE85}, Carrillo-Ni \cite{CN09} observed that, for a shrinker normalized as in \eqref{normalized_shrinker}, the constant 
\begin{align}\label{shrinker_entropy}
    \mu_g&:=2\Delta f-|\nabla f|^2+R+f-n
\end{align}
is equal to Perelman's functional $\mu(g,1)$ defined in \eqref{munudefinition}. 


\subsection{The triviality of the entropy on a Ricci shrinker}

Obviously, if we consider the canonical form $(M^n,g_t)_{t\in(-\infty,0)}$ of a Ricci shrinker $(M^n,g,f)$, where
\begin{align*}
    g_t:=-t\phi^*_{\log(-t)}g,\quad \partial_s\phi_s=-\nabla f\circ \phi_s,\quad \phi_0=\operatorname{id},
\end{align*}
then its asymptotic shrinker in the sense of Perelman \cite[Proposition 11.2]{Per02} is the shrinker itself. 

On the other hand, the Nash entropy on an ancient solution always converges to the entropy of its asymptotic (metric) soliton (c.f. \cite{Bam20c}, \cite{MZ21}). So if the ancient solution in question is the canonical form $(M,g_t)$, then the Nash entropy converges to the quantity $\mu_g$ defined in formula \eqref{shrinker_entropy}. We thus observe that the Nash entropy on the canonical form of a shrinker is always bounded from below. In particular, for any $x\in M$, we have
\begin{align}\label{boundednashforshrinker}
   \inf_{\tau>0} \N_{x,-1}(\tau)=\mu_g>-\infty.
\end{align}

Therefore, we may regard (the canonical form of) a Ricci shrinker as an ancient Ricci flow with bounded Nash entropy, and many results in \cite{Bam20a,Bam20c} can be applied in this case. This provides a perspective of understanding some classical results. Here we remind the reader that although in general a shrinker is not known to have bounded curvature, one may still apply the cut-off argument and heat kernel gaussian estimates of Li-Wang \cite{LW20} to carry out many of Bamler's proofs. Let us first of all recall the uniform Sobolev inequality of Li-Wang \cite{LW20}. In fact, a uniform Sobolev inequality is equivalent to a lower bound of Perelman's $\nu$-functional defined as follows
\begin{align}\label{munudefinition}
    \nonumber\bW(g,u,\tau)&:=\int_M\tau\left(\frac{|\nabla u|^2}{u}+Ru\right)dg-\int_Mu\log udg-n-\frac{n}{2}\log(4\pi\tau),
    \\
    \mu(g,\tau)&:=\inf\left\{\bW(g,u,\tau)\,\left|\, u\geq 0,\ \sqrt{u}\in C_0^\infty(M),\ \int_M udg=1\right\}\right.,
    \\\nonumber
    \nu(g)&:= \inf_{\tau>0}\mu(g,\tau).
\end{align}

\begin{Theorem}[\cite{LW20}]\label{LWSobolev}
Let $(M^n,g,f)$ be a shrinker normalized as in \eqref{normalized_shrinker}. Then we have
$$\nu(g):=\inf_{\tau>0}\mu(g,\tau)=\mu_g,$$
where $\mu_g$ is the constant in \eqref{shrinker_entropy}.
\end{Theorem}

We briefly describe the idea of the proof. One may consider the canonical form $(M^n,g_t)_{t\in(-\infty,0)}$ and estimate $\nu(g_{-1})$. Let $\tau>0$ be any number and let $u$ be any test function for $\mu(g_{-1},\tau)$. Define
\begin{align*}
    u_t(x):=\int_MK(x,t\,|\,\cdot,-1)u\,dg_{-1},\quad \tau_t=\tau-1-t,\quad t\le-1,
\end{align*}
where $K$ is the conjugate heat kernel. Then we have 
\begin{align}\label{monotonicityandconvergenceofbW}
    \frac{d}{dt}\bW(g_t,u_t,\tau_t)\ge0,\quad\lim_{t\to-1}\bW(g_t,u_t,\tau_t)=\bW(g,u,\tau),\quad \lim_{t\to-\infty}\bW(g_t,u_t,\tau_t)=\mu_g.
\end{align}
To see why the last equation is true, one may refer to \cite[Section 9]{CMZ21a}. In fact, when $-t\to\infty$, the difference between $-t$ and $\tau_t$ is negligible. Now, apparently, \eqref{monotonicityandconvergenceofbW} implies Theorem \ref{LWSobolev}. To make the argument rigorous, one may find the cut-off argument and heat kernel gaussian bounds in \cite{LW20} helpful.

\subsection{Upper volume growth estimate}

Cao-Zhou \cite{CZ10} proved the following nice growth estimates for the shrinker potential $f$ (see also \cite{HM11}). These estimates show that, to estimate the volume growth rate, one may consider the sub-level sets of $f$ instead of geodesic disks; the former proves to be much more analytically tractable than the latter.

\begin{Theorem}[\cite{CZ10}]
Let $(M^n,g,f)$ be a Ricci shrinker normalized as in \eqref{normalized_shrinker}, then we have
\begin{align*}
    \frac{1}{4}\left(\dist(x,p)-5n\right)_+^2\le f(x)-\mu_g\le\frac{1}{4}\left(\dist(x,p)+\sqrt{2n}\right)^2,
\end{align*}
where $p$ is the point where $f$ attains its minimum and $\mu_g$ is the quantity in \eqref{shrinker_entropy}.
\end{Theorem}

Following the idea mentioned at the beginning of this subsection, let $\mathcal{V}(s)$ and $\mathcal{R}(s)$ be the functions
\begin{equation}\label{R and V}
\mathcal{V}(s)=\int_{\{f< \mu_g+s\}}\,dg \quad \text{  and } \quad \mathcal{R}(s)=\int_{\{f< \mu_g+s\}} R\,dg.
\end{equation}
By integrating the trace of \eqref{normalized_shrinker}, Cao-Zhou \cite{CZ10} proved the following differential equation for $\mathcal{V}$ and $\mathcal{R}$
\begin{equation}\label{volumeode}
    \tfrac{n}{2}\mathcal{V}(s)-\mathcal{R}(s)=s\mathcal{V}'(s)-\mathcal{R}'(s).
\end{equation}
Using the above equation, Cao-Zhou \cite{CZ10} and Munteanu \cite{Mun09} showed that a complete gradient shrinker has at most Euclidean volume growth (see also \cite{HM11}). The volume estimate is sharp on the Gaussian shrinker and on the conical shrinker constructed by Felman-Ilmanen-Knopf \cite{FIK03}. See also the generalization by Munteanu-Wang \cite{MW14} to smooth metric measure space with $\Ric_f\ge 1/2$. 

\begin{Theorem}\cite{CZ10, Mun09, MW14}\label{shrinker up thm} Let $(M^n ,g, f)$ be a complete shrinking gradient Ricci soliton normalized as in \eqref{normalized_shrinker}. Then there is a dimensional constant $c=c(n)$ such that for any $p\in M$ and $r>0$, 
\begin{equation}\label{shrinker volume up}
|B_r(p)|\le ce^{f(p)-\mu_g}r^n.
\end{equation}
\end{Theorem}

As pointed out in \cite{HM11}, if the center $p$ is chosen to be the minimum point of $f$, then it follows from \eqref{shrinker_entropy} that $f(p)-\mu_g\le n/2$ and the constant on the R.H.S. of \eqref{shrinker volume up} can be made dimensional. However, the recent work of Bamler \cite{Bam20a} implies a stronger result than Theorem \ref{shrinker up thm}. Since shrinkers all have nonnegative scalar curvature (c.f. \cite{Che09}), one may apply the result of \cite[Theorem 8.1]{Bam20a} to Ricci shrinkers and show that, based at any point, the volume growth rate has an Euclidean upper bound. One may use the techniques in \cite{LW20} to make this argument rigorous. Here we only present the following conclusion.

\begin{Theorem}[{\cite[Theorem 8.1]{Bam20a}}]
Let $(M^n,g,f)$ be a (not necessarily normalized) Ricci shrinker. Then for any point $x\in M$ and any $r>0$, we have
$$|B_r(x)|\le C(n)r^n,$$
where $C(n)$ is a dimensional constant.
\end{Theorem}

A common class of shrinkers which model the nondegenerate neck pinch of the Ricci flow is that of the cylindrical shrinkers. Their volumes grow much slower than the Euclidean space. Under the assumption of positive lower bound of the scalar curvature, S.-J. Zhang \cite{Zhs11} established an upper volume bound for gradient shrinkers (see also \cite{CN09} for upper estimates under the bounded and nonnegative Ricci curvature conditions, and \cite{Che12} under the assumption of integral upper bound of $R$). This estimate is optimal because of the cylindrical examples $\mathbb{R}^{n-k}\times \mathbb{S}^{k}$ for $k\ge 2$.

\begin{Theorem}\cite{Zhs11}\label{cy shrink bdd} Suppose that $(M^n, g, f)$ is a complete shrinking gradient Ricci soliton normalized as in \eqref{normalized_shrinker} with scalar curvature bounded from below by a positive constant $\varepsilon$, i.e., $R\ge \varepsilon$ on $M$. Then for all $p\in M$, there exists a constant $C$ such that for all $r\ge 0$
\[
|B_r(p)|\le Cr^{n-2\varepsilon}.
\]
\end{Theorem}

Ni \cite{Ni05} showed that any nontrivial complete gradient Ricci shrinker with nonnegative Ricci curvature has scalar curvature uniformly bounded from below by a positive constant. In particular, it follows from Theorem \ref{cy shrink bdd} that such a shrinker has zero asymptotic volume ratio. This was also proved by Carrillo-Ni \cite[Proposition 2.1]{CN09}.  

\subsection{Lower volume growth estimate}

Munteanu-Wang \cite{MW12} first proved that the volume growth rate of a shrinking gradient Ricci soliton is at least linear. This rate is optimal since it is fulfilled by the standard round cylinder $\mathbb{S}^{n-1}\times\mathbb{R}$. However, their estimation constant 
$C(n)e^{c_1 \mu_g}$,
where $c_1=c_1(n)>1$, is not optimal.  Later, by applying their uniform Sobolev inequality on shrinking solitons, Li-Wang \cite{LW20} improved the estimation constant of \cite{MW12}. 

\begin{Theorem}[Lower volume growth estimate for shrinkers]\label{Shrinker lower growth}
Let $(M^n, g, f)$ be a complete noncompact shrinking gradient Ricci soliton normalized as in \eqref{normalized_shrinker}. There exist dimensional constants $c_0=c_0(n)$ and $r_0=r_0(n)$ such that for all $r\ge r_0$, 
\[
|B_r(p)|\ge c_0e^{\mu_g} r.
\]
\end{Theorem}

We briefly summarize the proof of Theorem \ref{Shrinker lower growth}. Consider the set 
$D(s):=\{2\sqrt{f-\mu_g}<s\}$ instead and its volume $V(s):=|D(s)|$. Indeed, we have $V(s)=\V(s^2/4)$. If, by contradiction, the conclusion is not true, then one may integrate \eqref{volumeode} to obtain that the volume of the annulus $|D(s+1)\setminus D(s)|$ is very small in comparison to $e^{\mu_g}$ whenever $s$ is large enough, and thus $|D(s+1)\setminus D(s)|$ is much smaller than $e^{\frac{2\mu_g}{n}}|D(s+1)\setminus D(s)|^{\frac{n-2}{n}}$. On the other hand, the Sobolev inequality in Theorem \ref{LWSobolev}, when applied to a proper  cut-off function, implies that $|D(s+1)\setminus D(s)|^{\frac{n-2}{n}}$ is smaller than $Ce^{-\frac{2\mu_g}{n}}\left(|D(s+2)\setminus D(s-1)|+\int_{D(s+2)\setminus D(s-1)}Rdg\right)$. When everything is thus added together, one may conclude that the shrinker has finite volume. However, this cannot happen because of \cite[Lemma 6.2]{MW12}.

In view of Theorem \ref{shrinker up thm}, a shrinker has at most Euclidean volume growth. Therefore, it is interesting to see when it has maximal volume growth. Note that a shrinker has maximum volume growth if and only if it has positive asymptotic volume ratio (AVR)
\[
    {\rm AVR}(g) := \lim_{r\to \infty} \frac{|B_r(o)|}{r^n}.
\]
Chow-Lu-Yang \cite{CLY12} established an equivalent condition for a shrinker to have positive asymptotic volume ratio.
\begin{Theorem} \cite{CLY12} Let $(M^n,g,f)$ be a complete noncompact shrinking gradient Ricci soliton normalized as in \eqref{normalized_shrinker}. $M$ has positive asymptotic volume ratio if and only if   
\[
\int_{n+2}^\infty\frac{\mathcal{R}(s)}{s\mathcal{V}(s)}\,ds<\infty,
\]
where $\mathcal{R}(s)$ and $\mathcal{V}(s)$ are functions defined in \eqref{R and V}.
\end{Theorem}
\section{Volume estimates for steady gradient Ricci solitons}
In this section, we investigate the volume bounds for steady gradient Ricci solitons, i.e. a complete Riemannian manifold $(M, g)$ with a smooth potential function $f:M\to\mathbb{R}$ satisfying
\[
\Ric+\nabla^2 f=0.
\]
Steady Ricci solitons are natural generalizations of Ricci-flat manifolds. Therefore it is expected that they share similar nice volume bounds. However, the volume growth of a steady soliton in full generality is less clear than the shrinker case.

It was shown by Hamilton \cite{Ha95} that $\nabla\left(|\nabla f|^2+R\right)=0$. Hence a non-Ricci-flat steady soliton, by scaling the metric if necessary, can be normalized in the way that
\begin{equation}\label{steady_normalization}
    |\nabla f|^2+R=1.
\end{equation}
As in the case of Ricci shrinkers, by \cite{Che09}, a steady gradient Ricci soliton has nonnegative scalar curvature. 

Munteanu-Sesum \cite{MS13} showed the following volume growth estimates of steady gradient Ricci solitons. 
\begin{Theorem}\cite{MS13} Suppose that $(M^n, g, f)$ is a complete noncompact steady gradient Ricci soliton. Then for all $p\in M$, there exist positive constants $a$ and $c$ such that for all large $r\gg 1$, it holds that
\[
c^{-1}r\le \left|B_r(p)\right|\le ce^{a\sqrt{r}}.
\]
\end{Theorem}
As an application of the above estimates, in the non-Ricci flat case, Munteanu-Sesum \cite[Corollary 5.2]{MS13} established bounds for $$F:=-f.$$  Namely, for any $p\in M$, there exists a large positive $c$ such that for all large positive $r$,  
\begin{equation}\label{rtr error}
r-c\sqrt{r}\le \sup_{\partial B_r(p)} F \le r+c.
\end{equation}
Using the weighted volume comparison theorem for smooth metric measure space, Wu \cite{Wu13} proved a similar estimate and showed the weak decay of scalar curvature on complete noncompact steady gradient Ricci solitons, i.e., $\liminf_{x\to\infty}R=0$. Under some comparison condition on the function $f+r$, Wei-Wu \cite
{WW13} showed that the $\sqrt{r}$ term in \eqref{rtr error} can be improved to $\ln r$, and that $M^n$ has at most Euclidean volume growth, i.e. for all large $r$
\begin{equation}\label{ln error}
r-c\ln r\le F \le r+c; 
\end{equation}
\[
\left|B_r(p)\right|\le Cr^n.
\]

There are two generic types of scalar curvature decay for steady Ricci solitons, namely the linear decay $R\le Cr^{-1}$ and the exponential decay $R\le Ce^{-r}$.  We shall prove that, under the conditions of proper potential function and linear scalar curvature decay, the estimate for $F$ in \eqref{ln error} holds and consequently, $M$ has at most polynomial volume growth. The linear scalar curvature decay and proper potential function conditions are satisfied by a number of recent examples of steady soliton \cite{Ap17,Bry05,BDW15,Cao94,CoD20,DW09,DW11,Sto15,Win17}. 
\begin{Theorem}\label{pot est} Let $(M^n, g,f)$, where $n\ge 4$, be a complete noncompact non-Ricci-flat steady gradient Ricci soliton normalized as in \eqref{steady_normalization}. If the scalar curvature $R$ of $M$ decays at least linearly, i.e. $R\le C_1/(r+1)$ and $f$ is proper, then there is a positive constant $C$ such that 
\[
r-C\ln(r+1)-C\le F\le r+C \text{  on } M,
\]
where $F:=-f$ and $r(\cdot):=\dist(\cdot,p)$ is the distance function based at $p$.
\end{Theorem}
\begin{Remark}
The upper and lower estimates of $-f$ are sharp as, upon scaling,  $|r+f|=O(1)$ on the Hamilton's cigar soliton, and $|r+f|\sim \ln (r+1)$ on the Cao's soliton on $\mathbb{C}^n$ with $n\ge 2$ (see \cite{Cao94, WW13,CoD20}). Estimates on $F$ under Ricci curvature conditions were also obtained in \cite{CC12,CN09,CDM20}.
\end{Remark}
Using the volume estimate of Wei-Wu \cite[Remark 3.2]{WW13}, we have
\begin{Corollary}
Under the same assumption of the previous theorem,  $M$ has at most polynomial volume growth.
\end{Corollary}

\begin{proof}[Proof of Theorem \ref{pot est}] Let $F:=-f$.
Rewriting \eqref{steady_normalization} as
\begin{equation}\label{ham id st}
    |\nabla F|^2+R=1
\end{equation}
and applying the fact that $R\ge 0$ on a complete steady gradient Ricci soliton, we obtain the upper bound for $F$. 
By our assumptions and the upper estimate for $F$, there are large positive constants $\rho_0$ and $C_0$ such that on $\{F\ge \rho_0\}$
\begin{align}\label{somethinguseful}
|\nabla F|^2=1-R\ge 1/2\quad \text{  and }\quad R\le C_0/F.
\end{align}
Let $\Psi_s$ be the flow of the vector field $\nabla F/|\nabla F|^2=-\nabla f/|\nabla f|^2$ with $\Psi_{\rho_0}=\text{  id }$. Moreover for $s\ge \rho_0$, $\Psi_s$ maps $\{F\ge \rho_0\}$ to $\{F\ge s\}$ and hence $\Psi_s(x)$ exists for all $s\ge \rho_0$ for any $x \in \{F\ge \rho_0\}$. Let $\Gamma_s:=\{F=s\}$. Then $\Gamma_s$ is a smooth compact hypersurface in $M$ and when $\Psi_s$ is restricted on $\Gamma_{\rho_0}$, it becomes a diffeomorphism from $\Gamma_{\rho_0}$ to $\Gamma_{s}$. 
For all $z\in \{F\ge \rho_0\}$, we denote $F(z)$ by $s$. Then there is a $z_0\in \Gamma_{\rho_0}$ such that $\Psi_s(z_0)=z$ and for all $t\in [\rho_0, s]$,
\[
F\left(\Psi_t(z_0)\right)-F(z_0)=\int_{\rho_0}^t\langle\nabla F, \dot{\Psi}_{\tau}(z_0)\rangle\, d\tau = \int_{\rho_0}^t\langle\nabla F, \nabla F\rangle/|\nabla F|^2\, d\tau=t-\rho_0.      
\]
Hence $F\left(\Psi_t(z_0)\right)=t$. Using \eqref{ham id st}
\begin{equation}
    \begin{split}
      \dist\left(\Psi_s(z_0), z_0\right)&\le \int_{\rho_0}^s\frac{1}{|\nabla F|}\,d\tau=\int_{\rho_0}^s\frac{1}{|\nabla f|}\,d\tau\\
          &=s-\rho_0+\int_{\rho_0}^s\frac{1-\sqrt{1-R}}{\sqrt{1-R}}\,d\tau\\
          &\le s-\rho_0+\int_{\rho_0}^s 2-2\sqrt{1-R}\,d\tau\\
          &\le s-\rho_0+\int_{\rho_0}^s 2R\,d\tau\\
          &\le F(z)-\rho_0+\int_{\rho_0}^s 2C_0/F\,d\tau\\
          &= F(z)-\rho_0+\int_{\rho_0}^s 2C_0/\tau\,d\tau\\
          &= F(z)-\rho_0+2C_0\ln s-2C_0\ln \rho_0,\\
          \end{split}
\end{equation}
where we have applied \eqref{somethinguseful}. Hence
\begin{equation}
    \begin{split}
-f(z)=F(z)&\ge \dist (z,p)-\sup_{K} \dist (\cdot,p)-2C_0\ln s\\
&\ge \dist (z,p)-\sup_{K}\dist (\cdot,p)-2C_0\ln \left(\dist (z,p)+F(p)\right),\\
       \end{split}
\end{equation}
where $K$ is the compact set $\{F\le \rho_0\}$ and we have used that fact that $F(z)=s\le \dist (z,p)+F(p)$. This completes the proof of the theorem.
\end{proof}
For general complete non-compact K\"{a}hler manifold $M^m$ of complex dimension $m$ with nonnegative holomorphic bisectional curvature everywhere and positive bisectional curvature at some point, Chen-Zhu \cite{CZ05} showed that for all large $r$
\begin{equation}\label{std low bdd}
|B_r(p)|\ge cr^m.
\end{equation}
An integration by parts argument was used by Cui \cite{Cui16} to show that \eqref{std low bdd} also holds on any complete K\"{a}hler steady gradient Ricci soliton with positive Ricci curvature and with scalar curvature attaining its maximum. The lower bound is sharp on the Cao soliton on $\mathbb{C}^m$ with positive bisectional curvature, where $m\ge 2$ \cite{Cao94}. In real dimension three, the Hamilton-Ivey estimate implies that any complete ancient solution to the Ricci flow has nonnegative sectional curvature (see \cite{Che09} and the references therein). In particular, any three dimensional complete gradient shrinking or steady soliton has nonnegative sectional curvature. When the curvature is positive and the scalar curvature attains its maximum on $M$, the authors \cite{CMZ21b} established a sharp quadratic volume estimate on the steady soliton.
\begin{Theorem}\cite{CMZ21b} Suppose that $(M^3, g, f)$ is a three dimensional complete steady gradient Ricci solton with positive sectional curvature. Assume that the scalar curvature attains its maximum somewhere. Then M has quadratic volume growth, i.e., $|B_r(p)|\sim r^2$ for all $r\gg 1$.
\end{Theorem}

For a complete Riemannian manifold $(M^n,g)$, if $\Ric\ge 0$ outside a compact set,  
the asymptotic volume ratio (AVR)
\[
    {\rm AVR}(g) := \lim_{r\to \infty} \frac{|B_r(o)|}{r^n}
\]
is well defined and does not depend on the choice of the basepoint $o\in M$. 
By a dimension reduction argument, Chow, Deng and the second named author proved the following result in dimension 4.
\begin{Theorem}[Theorem 1.10 in \cite{CDM20}]
Suppose that $(M^4,g,f)$ is a four-dimensional non-Ricci-flat steady Ricci soliton with nonnegative Ricci curvature outside a compact set. If the scalar curvature decays uniformly, then ${\rm AVR}(g)=0.$
\end{Theorem}

In fact, if we assume that $\Ric\ge 0$ everywhere on $M$ and do not assume uniform scalar curvature decay, by Theorem 1.8 in \cite{CDM20}, we still have that ${\rm AVR}(g)=0.$


In higher dimensions, under some conditions on the curvature tensor, the second and third named authors showed that the asymptotic volume ratio of a non-collapsed steady soliton with $\Ric\ge 0$ is zero \cite{MZ21}. 
The proof relied on the existence of Perelman's asymptotic shrinkers (\cite[Section 11]{Per02}) and the authors followed the arguments of Ni \cite{Ni05}.
\begin{Theorem}[Corollary 5.2 in \cite{MZ21}] Let $(M^n, g, f)$ be a complete noncompact and non-collpased steady gradient Ricci soliton with nonnegative Ricci curvature. Furthermore, suppose that $M$ is non-flat and there exists a constant $C>0$ such that 
\begin{equation}\label{Rm vs R}
|{\Rm}|\le CR \quad\text{  on  } \quad M. 
\end{equation}
Then $M$ has zero asymptotic volume ratio, i.e. ${\rm AVR}(g)=0.$
\end{Theorem} 
\begin{Remark}
Recently, the authors \cite{CMZ21d} have showed that \eqref{Rm vs R} holds on any four-dimensional complete non-Ricci-flat steady soliton singularity model (see also \cite{MW15a, CC20, Ch19, CL21}). 
\end{Remark}

The authors and Bamler have found an optimal volume growth estimate for noncollapsed steady gradient Ricci solitons. Their notion of noncollapsing is defined in terms of the boundedness of Nash entropy. More precisely, a steady soliton $(M^n,g,f)$ is called noncollapsed, if there exists a point $p\in M$ and a nonpositive number $\mu_\infty$, such that the canonical form $(M,g_t)_{t\in(-\infty,\infty)}$ satisfies
\begin{equation}\label{eq: bdd nash}
    \inf_{\tau>0}\mathcal N_{p,0}(\tau)= \mu_\infty\quad\text{ for all }\quad \tau>0.
\end{equation}
Note that this notion of noncollapsing is different from (indeed, stronger than) the classical definition of Perelman. Nevertheless, it is more natural for application, and is proved to be equivalent to Perelman's definition in several special cases \cite{Zhy20,MZ21}. 

\begin{Theorem}[\cite{BCMZ21}]
Suppose that $(M^n,g,f)$ is a complete steady gradient Ricci soliton, normalized as in \eqref{steady_normalization} if $g$ is not Ricci-flat.  Assume that the canonical form $(M^n,g_t)_{t\in \mathbb{R}}$ satisfies \eqref{eq: bdd nash}. Additionally, assume that \textbf{either one} of the following conditions is true:
\begin{enumerate}[(1)]
    \item $(M^n,g_t)_{t\in\mathbb{R}}$ arises as a singularity model; or
    \item 
    $(M^n,g)$ has bounded curvature.
\end{enumerate}
Then
\[
c(n,\mu_\infty) r^{\frac{n+1}{2}}\le|B_r(o)|
\le 
C(n)r^n \quad\text{ for all }\quad r>\bar r(n,\mu_\infty),
\]
where $ \mu_\infty:=\inf_{\tau>0} \mathcal{N}_{o,0}(\tau)=\lim_{\tau\to\infty}\mathcal N_{o,0}(\tau)  >-\infty$  and $c(n,\mu_\infty)$
a positive constant of the form
\[
    c(n,\mu_\infty)
    = \frac{c(n)}{\sqrt{1-\mu_\infty}} e^{\mu_\infty}. 
\]
Furthermore, the upper bound is also true for all $r>0$ (instead of $r\geq\bar r(n,\mu_\infty)$).
\end{Theorem}

\section{Volume estimates for expanding gradient Ricci solitons}
In this section, we study the volume growth of expanding gradient Ricci solitons, i.e., a triple $(M^n,g,f)$ consisting of a Riemannian manifold $(M,g)$ and $f\in C^{\infty}(M)$ satisfying 
\[
\Ric+\nabla^2 f=-\tfrac{1}{2}g.
\]
They are analogues of Einstein manifolds with negative Ricci curvature. Hamilton \cite{Ha95} showed that $\nabla \left(|\nabla f|^2+R+f\right)=0$. By adding a constant to $f$, we may normalized the expander in the way that
\begin{align}\label{expander_normalization}
    |\nabla f|^2+R=-f.
\end{align}
It was proved by Pigola–Rimoldi–Setti \cite{PRS11} and S.-J. Zhang \cite{Zhs11} that the scalar curvature of any complete Ricci expander satisfies
\begin{equation}\label{low R generic}
R\ge -\frac{n}{2},
\end{equation}
with equality holds somewhere if and only if $M$ is Einstein. This inequality also follows by applying the lower estimate for the scalar curvature $R\ge -n/2t$ to the canonical form of the expander (see \cite{RFV4}).
\subsection{Upper volume growth estimate}
For general Ricci expanders, because of the trivial examples of hyperbolic spaces, one cannot expect a volume growth estimate with a growth rate slower than exponential. In this respect, Munteanu-Wang \cite{MW14} proved a sharp exponential volume bound for smooth metric measure spaces  with a lower bound of the Bakry-Emery Ricci curvature  $\Ric_f\ge \lambda$ and a linear upper bound of $|\nabla f|$; these assumptions are satisfied by expanding gradient solitons.
\begin{Theorem} \cite{MW14} Let $(M^n, g, f)$ be a complete expanding gradient Ricci soliton. Then there is a constant $C>0$ such that for all $r>0$
\begin{equation}
    |B_r(p)|\le Ce^{\sqrt{(n-1)}r}.
\end{equation}
\end{Theorem}

Without assumption on the potential function, the bound cannot be improved to the polynomial one as the hyberbolic space form $\mathbb{H}^n$ has exponential volume growth.

Chen-Deruelle \cite{ChD15} showed the existence of a conical structure on a complete noncompact Ricci expander with finite Riemann curvature ratio $\limsup_{x\to\infty}r^2|{\Rm}|<\infty$. In particular, the conical expander has Euclidean volume growth (see also \cite{Che12} for estimates under $\lim_{x\to\infty}r^2|{\Rm}|=0$ or integral decay condition on $R$).
\begin{Theorem} \cite{ChD15} Any complete noncompact expanding gradient Ricci soliton $(M^n,g,f)$ with $\limsup_{x\to\infty}r^2|{\Rm}|<\infty$ must be asymptotically conical in the Gromov-Hausdorff sense. In particular, it has Euclidean volume growth, namely, $|B_r(p)|\sim r^n$ for all $r\gg 1$.
\end{Theorem}

\subsection{Lower volume growth estimate}
It was proven by Hamilton that any complete gradient Ricci expander with positive Ricci curvature has positive asymptotic volume ratio (c.f. \cite[Proposition 9.26]{CLN06}). Under some lower bound of the scalar curvature, Ni-Carrillo \cite{CN09} proved the following monotonicity formula for the volume of expanding gradient Ricci soliton.
\begin{Theorem}\label{vol low exp}\cite{CN09}
Let $(M^n,g,f)$ be a complete noncompact expanding gradient Ricci soliton with scalar curvature bounded from below by a constant, i.e., $R\ge -\beta$, where $\beta\ge 0$. Then for any $p\in M$ and $r\ge r_0$, it holds that 
\[
|B_r(p)|\ge |B_{r_0}(p)|\left(\tfrac{r+\alpha}{r_0+\alpha}\right)^{n-2\beta},
\]
where $\alpha:=\sqrt{\beta-f(p)}$.
\end{Theorem}
\begin{Remark}
The lower estimate is sharp on the hyperbolic cylinder.
\end{Remark}

\subsection{Estimate under scalar curvature upper bound}

Munteanu-Wang \cite{MW12} showed that If $R\ge -(n-1)/2$, then the gradient Ricci expander either is connected at infinity or splits isometrically like $\mathbb{R}\times N$, where $N$ is a compact Einstein manifold. Munteanu-Wang \cite{MW15b} also proved that, a K\"ahler expander with proper potential function $f$ must be connected at infinity.

$E$ is an end of $M$ relative to a compact subset $K$ if it is an unbounded component of $M\setminus K$ \cite[Definition 20.2]{Li12}. In general, an expander may have more than one ends, even if it doesn't split. The ends of an expander can behave very differently from each other. Indeed, there is a three dimensional expanding gradient Ricci soliton with two ends, one being hyperbolic like, another being asymptotically conical \cite{R13} (see also \cite{BM15, MR20} for other examples). The potential function $f$ of the soliton is proper on the conical end and remains bounded on the hyperbolic end. Therefore it is natural to localize the volume estimates on those ends with proper potential $f$. We first established the following volume estimate for three dimensional expanding solitons.

\begin{Theorem}\label{3d exp end vol} Let $(M^3, g, f)$ be a three dimensional complete noncompact expanding gradient Ricci soliton and let $E$ be an end of $M$ relative to some compact subset. 
Suppose the following conditions are satisfied:
\begin{itemize}
    \item $\lim_{x\in E,x\to\infty}f(x)=-\infty$;
    \item scalar curvature $R$ is bounded on $E$.
\end{itemize}
Then there is a large positive constant $C$ such that for all $r\gg 1$
\[
|B_r(p)\cap E|\le Cr^3.
\]
Moreover, either one of the following holds for all large $r$:
\begin{enumerate}
\item $|B_r(p)\cap E|\ge C^{-1} r^3;$
\item $|B_r(p)\cap E|\le C^{-1} r.$
\end{enumerate}
\end{Theorem}
\begin{Remark}
The two cases hold on conical and cylindrical ends, respectively. Under the assumptions of Theorem 1.1, one may wonder if the end $E$ must either be conical or cylindrical, or neither.
\end{Remark}
Using the result by Ni-Carrillo \cite{CN09}, we also have
\begin{Corollary}\label{coro exp 3d} Let $(M^3, g, f)$ be a three dimensional complete noncompact expanding gradient Ricci soliton with bounded scalar curvature $R$ and proper potential function. Then it has at most Euclidean volume growth. Specifically, fixing any $p \in M$, there exists a large constant $C$ such that for all large $r>0$ 
\[
|B_r(p)|\le Cr^3.
\]
In particular, any $3$ dimensional complete noncompact gradient Ricci expander with $\lim_{x\to\infty}R=0$ ($f$ is automatically proper in this case) has Euclidean volume growth, namely, $|B_r(p)|\sim r^3$.
\end{Corollary} 
\begin{Remark}
The estimate is sharp on Gaussian soliton and the positively/negatively curved Bryant expanding soliton. In view of the hyperbolic space $\mathbb{H}^3$, we see that the properness of $f$ is essential and can not be removed.
\end{Remark}
In higher dimensions, we also have
%
\begin{Theorem}\label{higherdimensionalexpander}
Let $(M^n, g, f)$ be an $n$ dimensional complete noncompact expanding gradient Ricci soliton and let $E$ be an end of $M$, where $n\ge 4$. Suppose the following conditions are satisfied:
\begin{itemize}
    \item $\lim_{x\in E,x\to\infty}f(x)=-\infty$;
    \item There is a constant $L_0$ (not necessarily nonnegative) such that
    \begin{equation}\label{Rupbdd}
R\le L_0 \text{  on  } E.
\end{equation}
\end{itemize} Then for any $p \in M$, there exists a large constant $C$ such that for all large $r>0$ 
\[
|B_r(p)\cap E|\le Cr^{n+2L_0}.
\]
\end{Theorem}
\begin{Remark}\label{L0}
The estimate is sharp on the hyperbolic cylinder. It can be viewed as the expanding analog of the volume estimate of shrinker by S.-J. Zhang \cite{Zhs11}, as well as the reverse estimate of \cite{CN09} (see also Theorems \ref{cy shrink bdd} and \ref{vol low exp}). 
\end{Remark}
\begin{Corollary}
Any complete noncompact expanding gradient Ricci soliton with bounded scalar curvature and proper potential function $f$ has at most polynomial volume growth. 
\end{Corollary}
When $f$ is proper and the scalar curvature $R$ is bounded on an end $E$ of the expander $M$, then by \cite[Theorem 27.4]{RFV4} and by the same proof as in \cite[lemma 7]{Cha20}, the potential function $f$ satisifes the following estimates near the infinity of $E$,
\begin{equation}\label{exp pot est1}
  \tfrac{1}{5}\left(\dist(x,p)\right)^2  \le v(x)\le \tfrac{1}{4}\left(\dist(x,p)\right)^2+\dist(x,p)\sqrt{v(p)}+v(p),
\end{equation}
where $p$ is a fixed point on $M$, $\dist(\cdot,p)$ is the distance function based at $p$ and $v:=\tfrac{n}{2}-f$. The above estimates will be used in the proofs of Theorems \ref{3d exp end vol} and Theorem \ref{higherdimensionalexpander}, since it relates the geodesic balls and sub-level sets of $v$.
\begin{proof}[Proof of Theorem \ref{3d exp end vol}]($n=3$ case) 
Since the scalar curvature is bounded, by the same argument as in \cite[Theorem 11]{Cha20}, $|{\Rm}|$ is bounded on the end $E$. We may invoke the Shi's estimate to see that, $|\nabla^k{\Rm}|$ is also bounded  on $E$ for all $k$.
Let $v=\frac{n}{2}-f$. Then by the properness of $f$, the boundedness of $R$, \eqref{expander_normalization}, and \eqref{exp pot est1}, we have $\lim_{x\to\infty} v=\infty$ and $|\nabla v|^2=v-n/2-R>0$ outside a compact subset of $E$. Hence one can find large $s_0$ such that on $\{x\in E: v(x)\ge s_0\}$, it holds that $\nabla v\neq 0$. Moreover, the level sets $\Gamma_s:=\{x\in E: v(x)=s\}$ are compact smooth hypersurfaces with $N$ components for all $s\ge s_0$, where $N$ is a constant integer, and $\{x\in E: v(x)\ge s_0\}$ is diffeomrphic to $[0,\infty)\times \Gamma_{s_0}$. By abuse of notation, we denote the induced metric on $\Gamma_s$ by $g$ and consider the flow $\psi_s$ of the vector field $\frac{\nabla v}{|\nabla v|^2}$ with $\psi_{s_0}=\text{  id }$. When restricted on $\Gamma_{s_0}$, $\psi_s: \Gamma_{s_0}\longrightarrow \Gamma_{s}$ are diffeomorphisms for all $s\geq s_0$. Let $h_s$ be the pull back metric $\psi_s^*g$ on $\Gamma_{s_0}$. Then we have
\begin{eqnarray*}
\frac{\partial}{\partial s}h_s&=&\psi_s^*\mathcal{L
}_{\frac{\nabla v}{|\nabla v|^2}}g=\psi_s^*\left(\frac{2\nabla ^2 v}{|\nabla v|^2}\right)=\psi_s^*\left(\frac{2\Ric+g}{|\nabla v|^2}\right),
\end{eqnarray*}
where $\mathcal{L}_{\frac{\nabla v}{|\nabla v|^2}}$ is the Lie derivative with respect to $\frac{\nabla v}{|\nabla v|^2}$. 
We denote by $dh_s$ the volume form induced by the metric $h_s$, then by \eqref{expander_normalization} and choosing $s_0$ sufficiently large
\begin{equation}\label{vol form}
\begin{split}
\frac{\partial }{\partial s}dh_s&= \tfrac{1}{2}\tr_{h_s}\frac{\partial}{\partial s}h_s\, dh_s\\
&= \psi_s^*\left(\frac{R-\Ric(\n,\n)+1}{|\nabla v|^2}\right)\,dh_s\\
&= \psi_s^*\left(\frac{R-\Ric(\n,\n)+1}{s-3/2-R}\right)\,dh_s\\
&\leq \psi_s^*\left(\frac{R-\Ric(\n,\n)+1}{s}\right)\,dh_s+\frac{C}{s^2}\,dh_s,
\end{split}
\end{equation}
where $\n$ is the normal vector $\frac{\nabla v}{|\nabla v|}$. Note that the second fundamental form of $\Gamma_s$ is $\frac{\nabla^2 v}{|\nabla v|}=\frac{\Ric+g/2}{\sqrt{s-3/2-R}}=O(s^{-1/2})$. Let $\sigma_1\le\sigma_2$ be the principal curvatures of $\Gamma_s$. By the Gauss equation, we have 

\begin{equation}\label{Gauss eqn}
\begin{split}
  K_s&=R_{1221}+ \sigma_1\sigma_2,;\\
  2K_s&=R-2\Ric(\n,\n)+2\sigma_1\sigma_2\\
      &=R-2\Ric(\n,\n)+O(s^{-1}),
\end{split}
\end{equation}
where $K_s$ denotes the Gauss curvature of $\Gamma_s$. We have 
\begin{equation}\label{vol form2}
\begin{split}
\frac{\partial }{\partial s}dh_s &\leq \psi_s^*\left(\frac{R-\Ric(\n,\n)+1}{s}\right)\,dh_s+\frac{C}{s^2}\,dh_s\\
&\leq \psi_s^*\left(\frac{2K_s+\Ric(\n,\n)+1}{s}\right)\,dh_s+\frac{C'}{s^2}\,dh_s\\
\end{split}
\end{equation}
By virtue of $\nabla R=2\Ric(\nabla f)$, we have $\psi_s^*(\Ric(\n,\n))=-\frac{1}{2}\langle\nabla R, \nabla v\rangle /|\nabla v|^2=O(|\nabla R|/|\nabla v|)=O(s^{-1/2})$ . It follows from \eqref{vol form2} that
\[
\frac{\partial }{\partial s}dh_s\leq \psi_s^*\left(\frac{2K_s+1}{s}\right)\,dh_s+\frac{C''}{s^{3/2}}\,dh_s.
\]
By applying the Gauss Bonnet Theorem to each components of $\Gamma_s$. we see that
\[
A'(s)\le A(s)/s+C_0/s+C''A(s) s^{-3/2}. 
\]
where $A(s):=$ area of $\Gamma_s$. Let $Q(s):=s^{-1}e^{2C''s^{-1/2}}(A(s)+C_0)$. Then $Q$ satisfies $Q'(s)\le 0$. By integrating the differential inequality, we see that for all $s\ge s_0$,
\[
A(s)\le Q(s_0)s.
\]
By the coarea formula, 
\begin{eqnarray*}
|\{s_0\le v\le s\}\cap E|&=&\int_{s_0}^s\int_{\Gamma_{\rho}}\frac{1}{|\nabla v|}\,d\rho\\
&\le& C\int_{s_0}^s \rho^{-1/2}A(\rho)\,d\rho\\
&\le& 2CQ(s_0)s^{3/2}/3.
\end{eqnarray*}
Using \eqref{exp pot est1}, we see that $v\sim \dist^2(\cdot,p)$ for some point $p.$
Hence we can find constants $C_1$ and $C_2$ such that for all large $r$
\[
|B_r(p)\cap E|\le |\{ v\le C_1r^2\}\cap E|\le C_2r^3+|\{ v\le s_0\}\cap E|.
\]
This shows the Euclidean volume upper bound for the expander. Let $\Lambda:=\int_{\Gamma_s}2K_s$ which is a constant for all large $s$ by the Gauss Bonnet Theorem. We have the following two cases. 
\begin{enumerate}
    \item If there is a large $s_1\ge s_0$ such that for all $s\ge s_1$, we have $A(s)+\Lambda\le 1$.
    \item There is a sequence of $s_i\to\infty$ such that $A(s_i)+\Lambda> 1$ for all $i$.
\end{enumerate}
In Case $1$, by virtue of Coarea formula, 
\begin{eqnarray*}
|\{s_1\le v\le s\}\cap E|=\int_{s_1}^s\int_{\Gamma_{\rho}}\frac{1}{|\nabla v|}&\le& C\int_{s_1}^s \rho^{-1/2}A(\rho)\,d\rho\\
&\le& C\int_{s_1}^s \rho^{-1/2}(1+|\Lambda|)\,d\rho\\
&\le& 2C(1+|\Lambda|)\sqrt{s}.
\end{eqnarray*}
$|B_r(p)\cap E|\le Cr$ then follows from the fact that $v\sim \dist_g^2(\cdot,p)$ (see \eqref{exp pot est1}). We shall show that $|B_r(p)\cap E|\ge Cr^3$ if Case $2$ holds. Let $i$ be a large integer to be specified. By the similar computation as before, we have
\begin{equation*}
\begin{split}
\frac{\partial }{\partial s}dh_s &\ge \psi_s^*\left(\frac{R-\Ric(\n,\n)+1}{s}\right)\,dh_s-\frac{C}{s^2}\,dh_s\\
&\ge \psi_s^*\left(\frac{2K_s+\Ric(\n,\n)+1}{s}\right)\,dh_s-\frac{C'}{s^2}\,dh_s\\
&\ge \psi_s^*\left(\frac{2K_s+1}{s}\right)\,dh_s-\frac{C''}{s^{3/2}}\,dh_s
\end{split}
\end{equation*}
By integrating both sides over $\Gamma_{s_0}$ and applying the Gauss Bonnet Theorem,
\[
A'(s)\ge (A(s)+\Lambda)/s-C''A(s)s^{-3/2}. 
\]
Let $W:=s^{-1}e^{-2C"s^{-1/2}}(A(s)+\Lambda)$. The derivative of $W$ satisfies 
\begin{equation*}
\begin{split}
W'(s)\ge -C''|\Lambda|s^{-5/2}e^{-2C"s^{-1/2}}\ge -C''|\Lambda|s^{-5/2}.
\end{split}
\end{equation*}
By the definition of $s_i$, $A(s_i)+\Lambda>1$. Since $\lim_{i\to\infty}s_i=\infty$, there exists a large $i$ such that
\[
e^{-2C''s_i^{-1/2}}-2C''|\Lambda|s_i^{-1/2}>1/2.
\]
Fixing this $i$, we may integrate the above inequality on $[s_i,s]$ and get
\[
W(s)\ge W(s_i)-2C''|\Lambda|s_i^{-3/2}\ge s_i^{-1}e^{-2C"s_i^{-1/2}}-2C''|\Lambda|s_i^{-3/2}\ge (2s_i)^{-1}.
\]
Hence for all $s\ge s_i(4|\Lambda|+1)$,
\[
A(s)\ge (2s_i)^{-1}s-\Lambda\ge (4s_i)^{-1}s.
\]
$|B_r(p)\cap E|\ge Cr^3$ is a consequence of a similar argument using Coarea formula and \eqref{exp pot est1}.
This completes the proof of the theorem.
\end{proof}

\begin{proof}[Proof of Corollary \ref{coro exp 3d}]
The first statement is a consequence of Theorem \ref{3d exp end vol} if we let $E=M$. The second assertion follows essentially from the argument in \cite[Proposition 5.1]{CN09} and Theorem \ref{3d exp end vol}. We include the argument for the sake of completeness. Since $R\to 0$ as $x\to \infty$, $f$ is proper \cite[Corollary 3]{Cha20} and there is a large $r_0>0$ such that on $M\setminus B_{r_0}(p)$
\[
R\ge -\tfrac{1}{2}.
\]
We integrate the trace of the expanding soliton equation over $B_{r}(p)$ to get
\begin{equation}
    \begin{split}
        \left|B_{r}(p)\setminus B_{r_0}(p)\right|&\le \int_{B_{r}(p)\setminus B_{r_0}(p)}(R+\tfrac{3}{2})\, dg\\
        &\le \int_{B_{r}(p)\setminus B_{r_0}(p)}(R+\tfrac{3}{2})\, dg+\int_{B_{r_0}(p)}(R+\tfrac{3}{2})\, dg\\
        &=-\int_{B_r(p)}\Delta f\,dg =-\int_{\partial B_r(p)}\partial_r f\\
        &\le \left(\tfrac{r}{2}+\sqrt{\tfrac{3}{2}-f(p)}\right)A\left(\partial B_r(p)\right),
    \end{split}
\end{equation}
we have used the facts that $|\nabla f|\le \dist (x,p)/2+\sqrt{3/2-f(p)}$ \cite[Theorem 27.4]{RFV4} and $R\ge -\frac{3}{2}$ on $M$ \cite{PRS11, Zhs11} (see also \eqref{low R generic}). Integrating the above inequality, we have that, for all $r\ge r_0+1$
\[
\left|B_{r_0+1}(p)\setminus B_{r_0}(p)\right|\left(\tfrac{r+2\sqrt{3/2-f(p)}}{r_0+1+2\sqrt{3/2-f(p)}}\right)^2\le \left|B_{r}(p)\setminus B_{r_0}(p)\right|.
\]
Hence $\left|B_{r}(p)\right|$ grows at least quadratically in $r$ and the cubic volume growth now follows from the dichotomy in Theorem \ref{3d exp end vol}. This completes the proof of the corollary. 
\end{proof}

In higher dimensions, in particular when $n\ge 4$, the Gauss Bonnet Theorem cannot be applied to the three dimensional level sets. A different argument is needed to handle the integral of the scalar curvature over the level set.

\begin{proof}[Proof of Theorem \ref{higherdimensionalexpander}]($n\ge 4$ case) 
We retain the same notation as in the proof of Theorem \ref{3d exp end vol}. We may choose $s_0$ large enough such that on $\{x\in E: v(x)\ge s_0\}$, it holds that $v-n/2-R\ge s_0-n/2-L_0>0$. Moreover, using the facts that $R$ is bounded on $E$ and $\psi_s^*(\Ric(\n,\n))=-\frac{1}{2}\langle\nabla R, \nabla v\rangle /|\nabla v|^2=-\partial_s \left(\psi_s^*R/2\right)$, we may proceed as in the proof of Theorem \ref{3d exp end vol} to see that

\begin{equation}
\begin{split}
\frac{\partial }{\partial s}dh_s&= \psi_s^*\left(\frac{R-\Ric(\n,\n)+(n-1)/2}{s-n/2-R}\right)\,dh_s\\
&= \psi_s^*\left(\frac{R+(n-1)/2}{s-n/2-R}\right)\,dh_s+ \psi_s^*\left(\frac{1}{2s-n-2R}\right)\partial_s \psi_s^*R\,dh_s\\
&= \psi_s^*\left(\frac{2R+(n-1)}{2s-n-2R}\right)\,dh_s+ \psi_s^*\left(\frac{1}{2s-n-2R}\right)\,dh_s\\
&\quad -\psi_s^*\left(\frac{1}{2s-n-2R}\right)\partial_s \psi_s^*(s-n/2-R)\,dh_s\\
&\le \psi_s^*\left(\frac{2L_0+n}{2s-n-2R}\right)\,dh_s-\tfrac{1}{2}\partial_s \psi_s^*\ln(2s-n-2R)\,dh_s.\\
\end{split}
\end{equation}
Let $J(s):=\int_{\Gamma_{s_0}}\psi_s^*\left(\sqrt{2s-n-2R}\right)\,dh_s\ge 0$. Thanks to \eqref{low R generic} and \eqref{Rupbdd}, it holds that $L_0+n/2\ge R+n/2\ge 0.$ Hence
\begin{equation}
\begin{split}
J'(s)&\le \int_{\Gamma_{s_0}}\psi_s^*\left(\frac{2L_0+n}{\sqrt{2s-n-2R}}\right)\,dh_s\\
&\le \frac{L_0+n/2}{s}J(s)+\frac{C}{s^2}J(s).
\end{split}
\end{equation}
By integrating the above inequality, we have for large $s\gg s_0$,
\[
\sqrt{s}A(s)\le J(s)\le s^{(L_0+n/2)}s_0^{-(L_0+n/2)}e^{C/s_0-C/s}J(s_0).
\]
Hence $A(s)=O(s^{(L_0+(n-1)/2)})$. 
In view of the rigidity case of \eqref{low R generic}, we see that $L_0+n/2> 0$. Otherwise, $R=n/2$ somewhere on $M$ and consequently, $\Ric\equiv -g/2$ and $\nabla^2 f\equiv 0$ on $M$ (see \cite{PRS11, Zhs11, RFV4}). Hence either $f$ is a constant or $M$ splits off a factor of $\mathbb{R}$. The former case is impossible as $f$ is proper along the end $E$. Contradiction also arises in the latter case as $\Ric\equiv -g/2$. Hence there is a large $s_1$ such that for all $s\ge s_1$
\begin{eqnarray*}
|\{s_1\le v\le s\}\cap E|=\int_{s_1}^s\int_{\Gamma_{\rho}}\frac{1}{|\nabla v|}&\le& C\int_{s_1}^s \rho^{-1/2}A(\rho)\,d\rho\\
&\le& C'\int_{s_1}^s \rho^{(L_0+(n-2)/2)}\,d\rho\\
&\le& C''s^{(L_0+n/2)}.
\end{eqnarray*}
We used the fact that $L_0+n/2> 0$ in the last inequality. $|B_r(p)\cap E|\le Cr^{2L_0+n}$ then follows similarly from \eqref{exp pot est1} as in the proof of Theorem \ref{3d exp end vol}. This finishes the the proof of Theorem \ref{higherdimensionalexpander}.
\end{proof}

\bibliographystyle{amsalpha}

\newcommand{\etalchar}[1]{$^{#1}$}
\providecommand{\bysame}{\leavevmode\hbox to3em{\hrulefill}\thinspace}
\providecommand{\MR}{\relax\ifhmode\unskip\space\fi MR }
\providecommand{\MRhref}[2]{%
  \href{http://www.ams.org/mathscinet-getitem?mr=#1}{#2}
}
\providecommand{\href}[2]{#2}

\bigskip
\bigskip

\noindent Department of Mathematics, University of California, San Diego, CA, 92093
\\ E-mail address: \verb"pachan@ucsd.edu "
\\

\noindent Department of Mathematics, University of California, San Diego, CA, 92093
\\ E-mail address: \verb"zim022@ucsd.edu"
\\

\noindent School of Mathematics, University of Minnesota, Twin Cities, MN, 55414
\\ E-mail address: \verb"zhan7298@umn.edu"

\end{document}